\pdfoutput=1

\documentclass[11pt]{article}
\usepackage[margin=1in]{geometry}
\usepackage{times}
\usepackage{amsmath,amsthm,amssymb}
\usepackage{graphicx}
\usepackage{mathrsfs}
\usepackage{paralist}
\usepackage[small,bf]{caption}
\usepackage{color}
\usepackage[colorlinks,linkcolor=black,bookmarksopen,
bookmarksnumbered,citecolor=black,urlcolor=black]{hyperref}

\graphicspath{{figs/cat/}}

\def\st{\, |\,}

\def\R{\mathbb{R}}

\def\Z{\mathbb{Z}}

\def\Q{\mathbb{Q}}

\DeclareMathOperator{\im}{im}

\DeclareMathOperator{\Hom}{Hom}
\DeclareMathOperator{\Ext}{Ext}

\DeclareMathOperator{\boundary}{\partial}

\newcommand{\rltvbndry}[3]{\operatorname{\boundary}_{#1}^{\, (#2,#3)}}

\newcommand{\abs}[1]{\lvert#1\rvert}
\newcommand{\norm}[1]{\lVert#1\rVert}

\def\p{\phantom{-}}

\newtheorem{theorem}{Theorem}[section]
\newtheorem*{theorem*}{Theorem}

\newtheorem*{proposition*}{Proposition}
\newtheorem{lemma}[theorem]{Lemma}
\newtheorem*{lemma*}{Lemma}
\newtheorem{claim}[theorem]{Claim}
\newtheorem*{claim*}{Claim}

\newtheorem*{axiom*}{Axiom}

\newtheorem*{conjecture*}{Conjecture}
\newtheorem{corollary}[theorem]{Corollary}
\newtheorem*{corollary*}{Corollary}

\theoremstyle{definition}
\newtheorem{definition}[theorem]{Definition}
\newtheorem*{definition*}{Definition}

\newtheorem*{example*}{Example}

\newtheorem*{exercise*}{Exercise}
\newtheorem*{recall*}{Recall}

\theoremstyle{remark}

\newtheorem*{note*}{Note}
\newtheorem{remark}[theorem]{Remark}
\newtheorem*{remark*}{Remark}

\newtheorem*{notation*}{Notation}
\newtheorem*{question*}{Question}

\newtheorem*{fact*}{Fact}

\newcommand{\bb}{\mathbf{b}}
\newcommand{\cc}{\mathbf{c}}
\newcommand{\ff}{\mathbf{f}}
\newcommand{\uu}{\mathbf{u}}
\newcommand{\vv}{\mathbf{v}}

\newcommand{\xx}{\mathbf{x}}
\newcommand{\yy}{\mathbf{y}}
\newcommand{\zz}{\mathbf{z}}

\definecolor{darkgrn}{rgb}{0, 0.75, 0}

\begin{document}

\title{Optimal Homologous Cycles, Total Unimodularity, \\
  and Linear Programming\thanks{A preliminary version
of this paper appeared in the Proceedings of 42nd ACM
Symposium on Theory of Computing, 2010.}}

\author{
Tamal K.~Dey\thanks{
Department of Computer Science and Engineering,
The Ohio State University, Columbus, OH 43210, USA.
\href{mailto:tamaldey@cse.ohio-state.edu}{\tt tamaldey@cse.ohio-state.edu},
\quad\url{http://www.cse.ohio-state.edu/\string~tamaldey}}
\quad\quad
Anil N.~Hirani\thanks{
Department of Computer Science,
University of Illinois at Urbana-Champaign, IL 61801, USA.
\href{mailto:hirani@cs.illinois.edu}{\tt hirani@cs.illinois.edu},
\quad\url{http://www.cs.illinois.edu/hirani}}
\quad\quad Bala Krishnamoorthy\thanks{
Department of Mathematics,
Washington State University, Pullman, WA 99164, USA.
\href{mailto:bkrishna@math.wsu.edu}{\tt bkrishna@math.wsu.edu},
\quad\url{http://www.math.wsu.edu/math/faculty/bkrishna}}
}

\date{}
\maketitle

\begin{abstract}
  Given a simplicial complex with weights on its simplices, and a
  nontrivial cycle on it, we are interested in finding the cycle with
  minimal weight which is homologous to the given one.
  Assuming that the homology is defined with integer ($\Z$)
  coefficients, we show the following
  (Theorem~\ref{thm:TUiff_trsnfr}):  
  \begin{quote}
    \emph{For a finite simplicial complex $K$ of dimension greater than
      $p$, the boundary matrix $[\boundary_{p+1}]$ is totally unimodular if
      and only if $H_p(L, L_0)$ is torsion-free, for all pure
      subcomplexes $L_0, L$ in $K$ of dimensions $p$ and $p+1$
      respectively, where $L_0 \subset L$.}
  \end{quote}
  Because of the total unimodularity of the boundary matrix, we can
  solve the optimization problem, which is inherently an integer
  programming problem, as a linear program and obtain an integer
  solution. Thus the problem of finding optimal cycles in a given
  homology class can be solved in polynomial time. This result is
  surprising in the backdrop of a recent result which says that the
  problem is NP-hard under $\Z_2$ coefficients which, being a
  \emph{field}, is in general easier to deal with. Our result implies,
  among other things, that one can compute in polynomial time an
  optimal $(d-1)$-cycle in a given homology class for any
  triangulation of an orientable compact $d$-manifold or for any
  finite simplicial complex embedded in $\R^d$. Our optimization
  approach can also be used for various related problems, such as
  finding an optimal \emph{chain} homologous to a given one when these
  are not cycles. Our result can also be viewed as providing a
  topological characterization of total unimodularity.
\end{abstract}

\section{Introduction} \label{sec:introduction} 

Topological cycles in shapes embody their important features. As a
result they find applications in scientific studies and engineering
developments. A version of the problem that often appears in practice
is that given a cycle in the shape, compute the shortest cycle in the
same topological class (homologous). For example, one may generate a
set of cycles from a simplicial complex using the persistence
algorithm~\cite{EdLeZo2002} and then ask to tighten them while
maintaining their homology classes. For two dimensional surfaces, this
problem and its relatives have been widely studied in recent years;
see, for
example,~\cite{ChVeErLaWh2008,ChErNa2009,ChFr2010,VeEr2006,DeLiSuCo2008}.
A natural question is to consider higher dimensional spaces which
allow higher dimensional cycles such as closed surfaces within a three
dimensional topological space. High dimensional applications
arise, for example, in the modeling of sensor networks
by Vietoris-Rips complexes of arbitrary dimension \cite{SiGh2007,
  TaJa2009}. Not surprisingly, these generalizations are hard to
compute which is confirmed by a recent result of Chen and Freedman
~\cite{ChFr2010a}. Notwithstanding this negative development, our
result shows that optimal homologous cycles in any finite dimension
are polynomial time computable for a large class of shapes if
homology is defined with integer coefficients.

Let $K$ be a simplicial complex. Informally, a $p$-cycle in $K$ is a
collection of $p$-simplices whose boundaries cancel mutually. One may
assign a non-zero weight to each $p$-simplex in $K$ which induces a
weighted 1-norm for each $p$-cycle in $K$. For example, the weight of
a $p$-simplex could be its volume. Given any $p$-cycle $c$ in $K$, our
problem is to compute a $p$-cycle $c^\ast$ which has the minimal
weighted 1-norm in the homology class of $c$. If some of the weights
are zero the problem can still be posed and solved, except that one may
not call it weighted 1-norm minimization. The homology classes are
defined with respect to coefficients in an abelian group such as $\Q$,
$\R$, $\Z$, $\Z_n$ etc. Often, the group $\Z_2$ is used mainly because
of simplicity and intuitive geometric interpretations.

Chen and Freedman~\cite{ChFr2010a} show that under $\Z_2$
coefficients, computing an optimal $p$-cycle $c^\ast$ is NP-hard for
$p\geq 1$. Moreover, their result implies that various relaxations may
still be NP-hard. For example, computing a constant factor
approximation of $c^\ast$ is NP-hard. Even if the rank of the
$p$-dimensional homology group is constant, computing $c^\ast$ remains
NP-hard for $p\geq 2$. The only settled positive case is a result of
Chambers, Erickson, and Nayyeri~\cite{ChErNa2009} who show that
computing optimal homologous loops for surfaces with constant genus is
polynomial time solvable though they prove the problem is NP-hard if
the genus is not a constant.

The above negative results put a roadblock in computing optimal
homologous cycles in high dimensions. Fortunately, our result shows
that it is not so hopeless -- if we switch to the coefficient group
$\Z$ instead of $\Z_2$, the problem becomes polynomial time solvable
for a fairly large class of spaces. This is a little surprising given
that $\Z$ is not a \emph{field} and so seems harder to deal with than
$\Z_2$ in general. For example, $\Z_2$-valued chains form a vector
space, but $\Z$-valued chains do not.

The problem of computing an optimal homologous cycle (or more
generally, chain) can be cast as a linear optimization problem.
Consequently, the problem becomes polynomial time solvable if the
homology group is defined over the reals, since it can be solved by
linear programming. Indeed this is the approach taken by Tahbaz-Salehi
and Jadbabaie~\cite{TaJa2009}. However, in general the optimal cycle
in that case may have fractional coefficients for its simplices, which
may be awkward in certain applications. One advantage of using $\Z$ is
that simplices appear with integral coefficients in the solution.  On
the other hand, the linear programming has to be replaced by integer
programming in the case of $\Z$.  Thus, it is not immediately clear if
the optimization problem is polynomial time solvable. One issue
in accommodating $\Z$ as the coefficient ring is that
integral coefficients other than $0$, $1$, or $-1$ do not have natural
geometric meaning.  Nevertheless, our experiments suggest that optimal
solutions in practice may contain coefficients only in $\{-1,0,1\}$.
Furthermore, as we show later, we can put a constraint in our
optimization to enforce the solution to have coefficients
in $\{-1,0,1\}$.

Our main observation is that the optimization problem that we
formulate can be solved by linear programming under certain
conditions, although it is inherently an integer programming problem.
It is known that a linear program provides an integer solution if and
only if the constraint matrix has a property called \emph{total
  unimodularity}. A matrix is totally unimodular if and only if each
of its square submatrices has a determinant of $0$, $1$, or $-1$. We
give a precise topological characterization of the complexes for which
the constraint matrix is totally unimodular. For this class of
complexes the optimal cycle can be computed in time polynomial in the
number of simplices. 
Totally unimodular matrices have a well-known geometric
characterization -- that the corresponding constraint polyhedron is
integral~\cite[Theorem 19.1]{Schrijver1986}. Our result provides a
topological characterization as well.

We can allow several variations to our problem because of our
optimization based approach. For example, we can probe into
intermediate solutions; we can produce the chain that bounds the
difference of the input and optimal cycles, and so forth. In fact, we
can also find an optimal chain homologous to a given one when the
chains are not cycles. In other words, we can leverage the flexibility
of the optimization formulation by linking results from two apparently
different fields, optimization theory and algebraic topology.

\section{Background}

Since our result bridges the two very different fields of algebraic
topology and optimization, we recall some relevant basic concepts and
definitions from these two fields.

\subsection{Basic definitions from algebraic topology}
Let $K$ be a finite simplicial complex of dimension greater than $p$.
A $p$-chain with $\Z$ coefficients in $K$ is a \emph{formal sum} of a
set of oriented $p$-simplices in $K$ where the sum is defined by
addition in $\Z$. Equivalently, it is an integer valued function on
the oriented $p$-simplices, which changes sign when the orientation is
reversed \cite[page 37]{Munkres1984}.

Two $p$-chains can be added by adding their values on corresponding
$p$-simplices, resulting in a group $C_p(K)$ called the $p$-chain
group of $K$. The \emph{elementary chain basis} for $C_p(K)$ is the
one consisting of integer valued functions that take the value 1 on a
single oriented $p$-simplex, $-1$ on the oppositely oriented simplex,
and 0 everywhere else. For an oriented $p$-simplex $\sigma$, we use
$\sigma$ to denote both the simplex and the corresponding elementary
chain basis element. The group $C_p(K)$ is free and abelian.  The
boundary of an oriented $p$-simplex $\sigma=[v_0,\ldots,v_p]$ is given
by
\[ \boundary_p \sigma = \sum_{i=0}^p
(-1)^i[v_0,..,\widehat{v_i},..,v_p]\, , 
\] 
where $\widehat{v}_i$ denotes that the vertex $v_i$ is to be deleted.
This function on $p$-simplices extends uniquely \cite[page
28]{Munkres1984} to the \emph{boundary operator} which is a
homomorphism: \[ \boundary_p\colon C_p(K)\rightarrow C_{p-1}(K)\, . \]

Like a linear operator between vector spaces, a homomorphism between
free abelian groups has a unique matrix representation with respect to
a choice of bases \cite[page~55]{Munkres1984}. The matrix form of
$\boundary_p$ will be denoted $[\boundary_p]$. Let
$\{\sigma_i\}_{i=0}^{m-1}$ and $\{\tau_j\}_{j=0}^{n-1}$ be the sets of
oriented $(p-1)$- and $p$-simplices respectively in $K$, ordered
arbitrarily. Thus $\{\sigma_i\}$ and $\{\tau_j\}$ also represent the
elementary chain bases for $C_{p-1}(K)$ and $C_p(K)$ respectively.
With respect to such bases $[\boundary_p]$ is an $m \times n$ matrix
with entries 0, 1 or $-1$. The coefficients of $\boundary_p \tau_j$ in
the $C_{p-1}(K)$ basis become the column $j$ (counting from 0) of
$[\boundary_p]$.

The kernel $\ker\boundary_p$ is called the group of $p$-{\em cycles}
and denoted $Z_p(K)$. The image $\im\boundary_{p+1}$ forms the group
of $p$-{\em boundaries} and denoted $B_p(K)$. Both $Z_p(K)$ and
$B_p(K)$ are subgroups of $C_p(K)$. Since $\boundary_p\circ
\boundary_{p+1} =0$, we have that $B_p(K) \subseteq Z_p(K)$, that is,
all $p$-boundaries are $p$-cycles though the converse is not
necessarily true. The $p$ dimensional homology group is the quotient
group $H_p(K)=Z_p(K)/B_p(K)$. Two $p$-chains $c$ and $c'$ in $K$ are
\emph{homologous} if $c=c'+\boundary_{p+1} d$ for some $(p+1)$-chain
$d$ in $K$. In particular, if $c=\boundary_{p+1} d$, we say $c$ is
homologous to zero. If a cycle $c$ is not homologous to zero, we call
it a \emph{non-trivial cycle}.

For a finite simplicial complex $K$, the groups of chains $C_p(K)$,
cycles $Z_p(K)$, and $H_p(K)$ are all finitely generated abelian
groups. By the fundamental theorem of finitely generated abelian
groups \cite[page 24]{Munkres1984} any such group $G$ can be written
as a direct sum of two groups $G=F\oplus T$ where $F\cong (\Z
\oplus\cdots\oplus \Z)$ and $T\cong(\Z/t_1\oplus\cdots\oplus \Z/t_k)$
with $t_i>1$ and $t_i$ dividing $t_{i+1}$. The subgroup $T$ is called
the \emph{torsion} of $G$. If $T=0$, we say $G$ is
\emph{torsion-free}.

Let $L_0$ be a subcomplex of a simplicial complex $L$. The quotient
group $C_p(L)/C_p(L_0)$ is called the group of \emph{relative chains}
of $L$ modulo $L_0$ and is denoted $C_p(L,L_0)$. The boundary operator
$\boundary_p\colon C_p(L)\rightarrow C_{p-1}(L)$ and its restriction
to $L_0$ induce a homomorphism
\[
\rltvbndry{p}{L}{L_0}\colon C_p(L,L_0)\rightarrow C_{p-1}(L,L_0)\, .
\]
As before, we have
$\rltvbndry{p}{L}{L_0}\circ\rltvbndry{p+1}{L}{L_0}=0$. Writing
$Z_p(L,L_0)={\rm ker}\rltvbndry{p}{L}{L_0}$ for \emph{relative cycles}
and $B_p(L,L_0)={\rm im}\rltvbndry{p+1}{L}{L_0}$ for
\emph{relative boundaries}, we obtain the \emph{relative homology
group} $H_p(L,L_0)=Z_p(L,L_0)/B_p(L,L_0)$. Sometimes, to distinguish
it from relative homology, the usual homology $H_p(L)$ is called the
{\em absolute homology group} of $L$.

\subsection{Total unimodularity and optimization}

Recall that a matrix is \emph{totally unimodular} if the determinant
of each square submatrix is $0$, $1$, or~$-1$. The significance of
total unimodularity in our setting is due to the following result:

\begin{theorem}{\rm \cite{VeDa1968}}\label{thm:polyhedron}
  Let $A$ be an $m\times n$ totally unimodular matrix and $\bb$ an
  integral vector, i.e., $\bb \in \Z^m$. Then the polyhedron
  $\mathcal{P} := \{\xx \in \R^n\;\st\; A\xx = \bb, \;
  \xx \ge \mathbf{0}\}$ is integral, meaning
  that $\mathcal{P}$ is the convex hull of the integral vectors
  contained in $\mathcal{P}$. In particular, the extreme points
  (vertices) of $\mathcal{P}$ are integral. Similarly the polyhedron
  $\mathcal{Q} := \{\xx \in \R^n\;\st\; A\xx \ge \bb\}$ is integral.

\end{theorem}

The following corollary shows why the above result is significant for
optimization problems. Consider an integral vector $\bb \in \Z^m$ and
a real vector of cost coefficients $\ff \in \R^n$. Consider the
\emph{integer} linear program%
\begin{equation}\label{opt:basic_ILP} \min
  \; \ff^T \xx \quad
  \text{subject to} \quad A \xx = \bb, \; \xx \ge \mathbf{0} \text{
    and } \xx \in \Z^n\, .
\end{equation}

\begin{corollary}\label{cor:TU_ILP}
  Let $A$ be a totally unimodular matrix. Then the integer linear
  program~\eqref{opt:basic_ILP} can be solved in time polynomial in
the dimensions of $A$. 
\end{corollary}

\begin{proof}
  Relax the integer linear program~\eqref{opt:basic_ILP} to a linear
  program by removing the integrality constraint $\xx \in \Z^n$. Then
  an interior point method for solving linear programs will find a
  real solution $\xx^\ast$ in polynomial time \cite{Schrijver1986} if
  it exists, and indicates the unboundedness or infeasibility of the
  linear program otherwise. In fact, since the matrix $A$ has entries
  0, 1 or $-1$, one can solve the linear program in strongly
  polynomial time \cite{Tardos1985, Tardos1986}. That is, the number
  of arithmetic operations do not depend on $\bb$ and $\ff$ and solely
  depends on the dimension of $A$. One still needs to show that the
  solution $\xx^\ast$ is integral.

  If the solution is unique then it lies at a vertex of the polyhedron
  $\mathcal{P}$ and thus it will be integral because of
  Theorem~\ref{thm:polyhedron}. If the optimal solution set is a face
  of $\mathcal{P}$ which is not a vertex then an interior point method
  may at first find a non-integral solution. However,
  by~\cite[Corollary 2.2]{BeTs1997} the polyhedron $\mathcal{P}$ must
  have at least one vertex. Then, by~\cite[Theorem 2.8]{BeTs1997} if
  the optimal cost is finite, there exists a vertex of $\mathcal{P}$
  where that optimal cost is achieved. Following the procedure
  described in~\cite{GuHeRoTeTs1993}, starting from the possibly
  non-integral solution obtained by an interior point method one can
  find such an integral optimal solution at a vertex in polynomial
  time.
\end{proof}

\section{Problem formulation}

Let $K$ be a finite simplicial complex of dimension $p$ or more. Given
an integer valued $p$-chain $x = \sum_{i=0}^{m-1} x(\sigma_i)\,
\sigma_i$ we use $\xx \in \Z^m$ to denote the vector formed by the
coefficients $x(\sigma_i)$. Thus, $\xx$ is the representation of the
chain $x$ in the elementary $p$-chain basis, and we will use $\xx$ and
$x$ interchangeably. For a vector $\vv \in \R^m$, the \emph{1-norm}
(or $\ell^1$-norm) $\norm{\vv}_1$ is defined to be $\sum_i \,
\abs{v_i}$.  Let $W$ be any real $m \times m$ diagonal matrix with
diagonal entries $w_i$. Then, the 1-norm of $W\, \vv$, that is,
$\norm{W\, \vv}_1$ is $\sum_i \, \abs{w_i}\abs{v_i}$. (If $W$ is a
general $m \times m$ nonsingular matrix then $\norm{W\, \vv}_1$ is
called the \emph{weighted 1-norm} of $\vv$.) The norm or weighted norm
of an integral vector $\vv \in \Z^m$ is defined by considering $\vv$
to be in $\R^m$. We now state in words the problem of optimal
homologous chains and later formalize it in~\eqref{opt:OHCP}:%
\begin{quote}
  Given a $p$-chain $\cc$ in $K$ and a diagonal matrix $W$ of
  appropriate dimension, the optimal homologous chain problem (OHCP)
  is to find a chain $\cc^\ast$ which has the minimal 1-norm $\norm{W
    \cc^\ast}_1$ among all chains homologous to~$\cc$.
\end{quote}

\begin{remark} 
  In the natural case where simplices are weighted and the optimality
  of the chains is to be determined with respect to these weights, we
  may take $W$ to be diagonal with $w_i$ being the weight of simplex
  $\sigma_i$. In our formulation some of the weights can be 0. Notice
  that the signs of the simplex weights are ignored in our formulation
  since we only work with norms.
\end{remark}

\begin{remark} 

  In Section~\ref{sec:introduction} we surveyed the computational
  topology literature on the problem of finding optimal homologous
  \emph{cycles}. The flexibility of our formulation allows us to solve
  the more general, optimal homologous \emph{chain} problem, with the
  cycle case being a special case requiring no modification in the
  equations, algorithm, or theorems.%
 \end{remark}

\begin{remark}
  The choice of 1-norm is important. At first, it might seem easier to
  pose OHCP using 2-norm. Then, calculus can be used to pose
    the minimization as a stationary point problem when OHCP is
    formulated with only equality constraints which appear in~\eqref{opt:OHCP}
    below. This case can be solved as a linear system of equations.
  By using 1-norm instead of 2-norm, we have to solve a linear program
  (as we will show below) instead of a linear system. But in return,
  we are able to get integer valued solutions when the appropriate
  conditions are satisfied.
\end{remark}

The formulation of OHCP is the \emph{weighted $\ell^1$-optimization}
of homologous chains. This is very general and allows for different
types of optimality to be achieved by choosing different weight
matrices.  For example, assume that the simplicial complex $K$ of
dimension greater than $p$ is embedded in $\R^d$, where $d \ge
p+1$. Let $W$ be a diagonal matrix with the $i$-th diagonal entry
being the Euclidean $p$-dimensional volume of a $p$-simplex. This
specializes the problem to the \emph{Euclidean $\ell^1$-optimization}
problem. The resulting optimal chain has the smallest $p$-dimensional
volume amongst all chains homologous to the given one. If $W$ is taken
to be the identity matrix, with appropriate additional conditions to
the above formulation, one can solve the \emph{$\ell^0$-optimization}
problem.  The resulting optimal solution has the smallest
\emph{number} of $p$-simplices amongst all chains homologous to $\cc$,
as we show in Section~\ref{subsec:L0}.

\emph{The central idea of this paper consists of the following steps}:
\begin{inparaenum}[(i)] \item write OHCP as an integer program
  involving 1-norm minimization, subject to linear constraints; \item
  convert the integer program into an integer \emph{linear} program by
  converting the 1-norm cost function to a linear one using the
  standard technique of introducing some extra variables and
  constraints; \item find the conditions under which the constraint
  matrix of the integer linear program is totally unimodular; and
  \item for this class of problems, relax the integer linear program
  to a linear program by dropping the constraint that the variables be
  integral. The resulting optimal chain obtained by solving the linear
  program will be an integer valued chain homologous to the given
  chain.
\end{inparaenum}
\subsection{Optimal homologous chains and linear programming}
Now we formally pose OHCP as an optimization problem. After showing
existence of solutions we reformulate the optimization problem as an
integer linear program and eventually as a linear program.

Assume that the number of $p$- and $(p+1)$-simplices in $K$ is $m$ and
$n$ respectively, and let $W$ be a diagonal $m \times m$ matrix.
Given an integer valued $p$-chain $\cc$ the optimal homologous chain
problem is to solve:%
\begin{equation} \label{opt:OHCP}
  \boxed{\underset{\xx,\, \yy}{\min} \, \norm{W\,\xx}_1 \quad
  \text{such that}\quad \xx = \cc + [\boundary_{p+1}] \; \yy, \text{
    and } \xx \in \Z^m, \;\yy \in \Z^n\, .}
\end{equation}

In the problem formulation~\eqref{opt:OHCP} we have given no
indication of the algorithm that will be used to solve the problem.
Before we develop the computational side, it is important to show that
a solution to this problem always exists.
\begin{claim} \label{clm:existence} For any given $p$-chain $\cc$ and
  any matrix $W$, the solution to problem~\eqref{opt:OHCP} exists.
\end{claim} 
\begin{proof}
  Define the set
  \[
  U_\cc := \{ \norm{W\, \xx}_1 \; \st \; \xx = \cc + [\boundary_{p+1}]\;
  \yy, \; \xx \in \Z^m \text{ and } \yy \in \Z^n \}\, .
  \]
  We show that this set has a minimum which is contained in the
  set. Consider the subset $U'_\cc \subseteq U_\cc$ defined by 
  \[
  U'_\cc
  = \{ \norm{W\, \xx}_1 \; \st \; \norm{W\, \xx}_1 \le \norm{W\,
    \cc}_1, \;\xx = \cc + [\boundary_{p+1}]\;
  \yy, \; \xx \in \Z^m \text{ and } \yy \in \Z^n \}\, .
  \]
  This set $U'_\cc$ is finite since $\xx$ is integral. Therefore,
  $\inf U_\cc = \inf U'_\cc = \min U'_\cc$.
\end{proof}

In the rest of this paper we assume that $W$ is a diagonal matrix
obtained from \emph{weights} on simplices as follows. Let $w$ be a
real-valued weight function on the oriented $p$-simplices of $K$ and
let $W$ be the corresponding diagonal matrix (the $i$-th diagonal
entry of $W$ is $w(\sigma_i) = w_i$).

The resulting objective function $\norm{W\, \xx}_1 = \sum_i \,
\abs{w_i} \, \abs{x_i}$ in~\eqref{opt:OHCP} is not linear in $x_i$
because it uses the absolute value of $x_i$. It is however,
piecewise-linear in these variables. As a result,~\eqref{opt:OHCP} can
be reformulated as an integer \emph{linear} program in the following
standard way~\cite[page 18]{BeTs1997}:%
\begin{align} \label{opt:ILP}
  \min \; &
  \sum_i \, \abs{w_i} \, (x_i^+ + x_i^-)\notag\\
  \text{subject to}\quad & \xx^+ - \xx^- = \cc + [\boundary_{p+1}]\;
  \yy\\
  &\xx^+, \; \xx^- \ge \mathbf{0}\notag\\
  &\xx^+, \, \xx^- \in \Z^m, \; \yy \in \Z^n \, . \notag
\end{align} 
  Comparing the above formulation to the standard form integer linear
  program in~\eqref{opt:basic_ILP}, note that the vector $\xx$
  in~\eqref{opt:basic_ILP} corresponds to $[\xx^+,\, \xx^-,\, \yy]^T$
  in~\eqref{opt:ILP} above. Thus the minimization is over $\xx^+$,
  $\xx^-$ and $\yy$, and the coefficients of $x^+_i$ and $x^-_i$ in
  the objective function are $\abs{w_i}$, but the coefficients
  corresponding to $y_j$ are zero. The linear programming relaxation
  of this formulation just removes the constraints about the variables
  being integral. The resulting linear program is:
\begin{align}  \label{opt:LP} 
  \min \;&  \sum_i \, \abs{w_i} \, (x_i^+ + x_i^-)\notag\\
  \text{subject to}\quad &
  \xx^+ - \xx^- = \cc + [\boundary_{p+1}]\; \yy\\
  & \xx^+, \; \xx^- \ge \mathbf{0} \, . \notag
\end{align}%
To use the result about standard form polyhedron in
Theorem~\ref{thm:polyhedron} we can eliminate the free (unrestricted
in sign) variables $\yy$ by replacing these by $\yy^+ - \yy^-$ and
imposing the non-negativity constraints on the new variables
\cite[page 5]{BeTs1997}. The resulting linear program has the same
objective function, and the equality constraints:%
\begin{equation} \label{opt:standard_LP}
  \xx^+ - \xx^- = \cc + [\boundary_{p+1}] \; (\yy^+ - \yy^-)\, ,
\end{equation}
and thus the equality constraint matrix is $\begin{bmatrix} I & -I &
  -B & B \end{bmatrix}$, where $B = [\boundary_{p+1}]$. We now prove a
result about the total unimodularity of this matrix.

\begin{lemma}\label{lem:BimpliesA_TU}
  If $B = [\boundary_{p+1}]$ is totally unimodular then so is the
  matrix $\begin{bmatrix} I & -I & -B & B \end{bmatrix}$.
\end{lemma}
\begin{proof}
  The proof uses operations that preserve the total unimodularity of a
  matrix. These are listed in \cite[page 280]{Schrijver1986}. If $B$ is
  totally unimodular then so is the matrix $\begin{bmatrix}-B &
    B\end{bmatrix}$ since scalar multiples of columns of $B$ are being
  appended on the left to get this matrix. The full matrix in
  question can be obtained from this one by appending columns with a
  single $\pm 1$ on the left, which proves the result.
\end{proof}
As a result of Corollary~\ref{cor:TU_ILP} and
Lemma~\ref{lem:BimpliesA_TU}, we have the following {\em algorithmic}
result.
\begin{theorem} \label{thm:TU_plynmltm}
  If the boundary matrix $[\boundary_{p+1}]$ of a finite simplicial
  complex of dimension greater than $p$ is totally
  unimodular, the optimal homologous chain problem~\eqref{opt:OHCP}
  for $p$-chains can be solved in polynomial time.
\end{theorem}
\begin{proof}
  We have seen above that a reformulation of OHCP \eqref{opt:OHCP},
  without the integrality constraints, leads to the linear
  program~\eqref{opt:LP}. By Lemma~\ref{lem:BimpliesA_TU}, the
  equality constraint matrix of this linear program is totally
  unimodular. Then by Corollary~\ref{cor:TU_ILP} the linear
  program~\eqref{opt:LP} can be solved in polynomial time, while
  achieving an integral solution.
\end{proof}

\begin{remark}
  One may wonder why Theorem~\ref{thm:TU_plynmltm} does not work when
  $\Z_2$-valued chains are considered instead of integer-valued
  chains. We could simulate $\Z_2$ arithmetic while using integers or
  reals by modifying~\eqref{opt:OHCP} as follows:
  \begin{equation} 
    \underset{\xx,\, \yy}{\min} \, \norm{W\,\xx}_1 \quad
    \text{such that}\quad \xx + 2\, \uu = \cc + [\boundary_{p+1}] \; \yy, \text{
      and } \xx \in \{0, 1\}^m, \; \uu \in \Z^m, \; \yy \in \Z^n\, .
  \end{equation}
  The trouble is that the coefficient $2$ of $\uu$ destroys the total
  unimodularity of the constraint matrix in the linear programming
  relaxation of the above formulation, even when $[\boundary_{p+1}]$
  is totally unimodular. Thus we cannot solve the above integer
  program as a linear program and still get integer solutions. 
\end{remark}

\begin{remark} \label{rmrk:bounding-chain}
  We can associate weights with $(p+1)$-simplices while formulating
  the optimization problem~\eqref{opt:OHCP}. Then, we could minimize
  $\norm{W\zz}_1$ where $\zz=[\xx,\yy]^T$.  In that case, we obtain a
  $p$-chain $c^*$ homologous to the given chain $c$ and also a
  $(p+1)$-chain $d$ whose boundary is $c^*-c$ and the weights of $c^*$
  and $d$ together are the smallest.  If the given cycle $c$ is null
  homologous, the optimal $y$ would be an optimal $(p+1)$-chain
  bounded by $c$.
\end{remark}

\begin{remark}
  The simplex method and its variants search only the basic feasible
  solutions (vertices of the constraint polyhedron), while choosing
  ones that never make the objective function worse. Thus if the
  polyhedron is integral, one could stop the simplex method at any
  step before reaching optimality and still obtain an integer valued
  homologous chain whose norm is no worse than that of the given
  chain.
\end{remark}

\subsection{Minimizing the number of simplices} 
\label{subsec:L0}

The general weighted $\ell^1$-optimization problem~\eqref{opt:OHCP}
can be specialized by choosing different weight matrices. One can also
solve variations of the OHCP problem by adding other constraints which
do not destroy the total unimodularity of the constraint matrix. We
consider one such specialization here -- that of finding a homologous
chain with the smallest \emph{number} of simplices.

If the matrix $W$ is chosen to be the identity matrix, then one can
solve the $\ell^0$-optimization problem by solving a modified version
of the $\ell^1$-optimization problem~\eqref{opt:OHCP}. One just
imposes the extra condition that every entry of $\cc$ and $\xx$ be in
$\{-1, 0, 1\}$. With this choice of $W = I$ and with $\cc \in \{-1, 0,
1\}^m$, the problem~\eqref{opt:OHCP} becomes:
\begin{equation}
  \label{opt:OHCP_L0}
  \underset{\xx,\, \yy}{\min} \, \norm{\xx}_1 \quad
  \text{such that}\quad \xx = \cc + [\boundary_{p+1}] \; \yy, \text{
    and } \xx \in \{-1, 0, 1\}^m, \;\yy \in \Z^n\, .
\end{equation}

\begin{theorem}\label{thm:L0_plynmltm}
  For any given $p$-chain $\cc \in \{-1, 0, 1\}^m$, a solution to
  problem~\eqref{opt:OHCP_L0} exists. Furthermore, amongst all $\xx$
  homologous to $\cc$, the optimal homologous chain $\xx^\ast$
  has the smallest number of nonzero entries, that is, it is the
  $\ell^0$-optimal homologous chain.
\end{theorem}
\begin{proof}
  The proof of existence is identical to the proof of
  Claim~\ref{clm:existence}. The condition that $\cc$ takes values in
  $-1$, 0, 1 ensures that at least $\xx = \cc$ can be taken as the
  solution if no other homologous chain exists. For the
  $\ell^0$-optimality, note that since the entries of the optimal
  solution $\xx^\ast$ are constrained to be in $\{-1, 0, 1\}$, the
  1-norm measures the number of nonzero entries. Thus the 1-norm
  optimal solution is also the one with the smallest number of
  non-zero entries.
\end{proof}

\begin{remark} \label{rem:hour_glass}
  Note that even with the given chain $\cc$ taking values in $\{-1, 0,
  1 \}$, without the extra constraint that $\xx \in \{-1, 0, 1\}^m$
  (rather than just $\xx \in \Z^m$), the optimal 1-norm solution
  components may take values outside $\{-1, 0, 1\}$. For example,
  consider the simplicial complex $K$ triangulating a cylinder which
  is shaped like an hourglass. Let $c_1$ and $c_2$ be the two boundary
  cycles of the hour glass so that $c_1+c_2$ is not trivial.  Let $z$
  be the smallest cycle around the middle of the hour glass which is
  homologous to each of $c_1$ and $c_2$. Since $c_1+c_2=2z$, the
  optimal cycle homologous to $c_1+c_2$ has values $2$ or $-2$ for
  some edges even if $c_1$ and $c_2$ have values only in $\{-1,0,1\}$
  for all edges. It may or may not be true that the number of nonzero
  entries is minimal in such an optimal solution. We have not proved
  it either way. But Theorem~\ref{thm:L0_plynmltm} provides a
  guarantee for computing $\ell^0$-optimal solution when the
  additional constraints are placed on $\xx$.
\end{remark}
The linear programming relaxation of problem~\eqref{opt:OHCP_L0} is
\begin{align}  \label{opt:LP_L0} 
  \min \;&  \sum_i \, (x_i^+ + x_i^-)\notag\\
  \text{subject to}\quad &
  \xx^+ - \xx^- = \cc + [\boundary_{p+1}]\; \yy\\
  &\xx^+,\; \xx^- \le \mathbf{1}\notag\\
  &\xx^+, \; \xx^- \ge\; \mathbf{0} \, . \notag
\end{align}%
One can show the integrality of the feasible set polyhedron by using
slack variables to convert the inequalities $\xx^+ \le \mathbf{1}$ and
$\xx^- \le \mathbf{1}$ to equalities and then using the $\mathcal{P}$
form of the polyhedron from Theorem~\ref{thm:polyhedron}.
Equivalently, all the constraints can be written as inequalities and
the $\mathcal{Q}$ polyhedron can be used. For a change we choose the
latter method here.  Writing the constraints as inequalities, in
matrix form the constraints are

\begin{equation} \label{eq:inequalities_L0}
\begin{bmatrix}
  -I & \p I & \p B & -B\\
  \p I & -I & -B & \p B\\
  -I & \p0 & \p0 & \p0\\
  \p0 & -I & \p0 & \p0\\
  \p I& \p0 & \p0 & \p0\\
  \p0 & \p I & \p0 & \p0\\
  \p0 & \p0 & \p I & \p0\\
  \p0 & \p0 & \p0 & \p I
\end{bmatrix} 
\begin{bmatrix}
  \xx^+ \\
  \xx^- \\
  \yy^+ \\
  \yy^-
\end{bmatrix} \ge
\begin{bmatrix}
  -\cc\\
  \p\cc\\
  -\mathbf{1}\\
  -\mathbf{1}\\
  \p\mathbf{0}\\
  \p\mathbf{0}\\
  \p\mathbf{0}\\
  \p\mathbf{0}
\end{bmatrix}\, ,
\end{equation}
where $B=[\boundary_{p+1}]$.  Then analogously to
Lemma~\ref{lem:BimpliesA_TU} and Theorem~\ref{thm:TU_plynmltm} the
following are true.
\begin{lemma}
  If $B = [\boundary_{p+1}]$ is totally unimodular then so is the
  constraint matrix in~\eqref{eq:inequalities_L0}.
\end{lemma}
\begin{theorem}
  If the boundary matrix $[\boundary_{p+1}]$ of a finite simplicial
  complex of dimension greater than $p$ is totally unimodular, then
  given a $p$-chain that takes values in $\{-1, 0, 1\}$, a homologous
  $p$-chain with the smallest number of non-zeros taking values in
  $\{-1,0,1\}$ can be found in polynomial time.
\end{theorem}

In subsequent sections, we characterize the simplicial complexes for
which the boundary matrix $[\boundary_{p+1}]$ is totally
unimodular. These are the main theoretical results of this paper,
formalized as Theorems~\ref{thm:manifolds_TU}, \ref{thm:TUiff_trsnfr},
and~\ref{thm:embedded}.

\section{Manifolds}
\label{sec:manifolds}

Our results in Section~\ref{sec:TUiff_trsnfr} are valid for \emph{any}
finite simplicial complex. But first we consider a simpler case --
simplicial complexes that are triangulations of manifolds. We show
that for finite triangulations of compact $p$-dimensional {\em
  orientable} manifolds, the top non-trivial boundary matrix
$[\boundary_p]$ is totally unimodular irrespective of the orientations
of its simplices. We also give examples of non-orientable manifolds
where total unimodularity does not hold. Further examination of why
total unimodularity does not hold in these cases leads to our main
results in Theorems~\ref{thm:TUiff_trsnfr}.

\subsection{Orientable manifolds}
Let $K$ be a finite simplicial complex that triangulates a
$(p+1)$-dimensional compact orientable manifold $M$. As before, let
$[\boundary_{p+1}]$ be the matrix corresponding to $\boundary_{p+1} :
C_{p+1}(K) \to C_p(K)$ in the elementary chain bases.

\begin{theorem} \label{thm:manifolds_TU}

  For a finite simplicial complex triangulating a $(p+1)$-dimensional
  compact orientable manifold, $[\boundary_{p+1}]$ is totally
  unimodular irrespective of the orientations of the simplices.

\end{theorem}
\begin{proof}
  First, we prove the theorem assuming that the $(p+1)$-dimensional
  simplices of $K$ are oriented \emph{consistently}. Then, we argue
  that the result still holds when orientations are arbitrary.

  Consistent orientation of $(p+1)$-simplices means that they are
  oriented in such a way that for the $(p+1)$-chain $c$, which takes
  the value 1 on each oriented $(p+1)$-simplex in $K$,
  $\boundary_{p+1} c$ is carried by the topological boundary
  $\boundary M$ of $M$. If $M$ has no boundary then $\boundary_{p+1}
  c$ is 0. It is known that consistent orientation of
  $(p+1)$-simplices always exists for a finite triangulation of a
  compact orientable manifold. Therefore, assume that the given
  triangulation has consistent orientation for the $(p+1)$-simplices.
  The orientation of the $p$- and lower dimensional simplices can be
  chosen arbitrarily.

  Each $p$-face $\tau$ is the face of either one or two
  $(p+1)$-simplices (depending on whether $\tau$ is a boundary face or
  not). Thus the row of $[\boundary_{p+1}]$ corresponding to $\tau$
  contains one or two nonzeros. Such a nonzero entry is 1 if the
  orientation of $\tau$ agrees with that of the corresponding
  $(p+1)$-simplex and $-1$ if it does not.

  Heller and Tompkins \cite{HeTo1956} gave a sufficient condition for
  the unimodularity of $\{-1,0,1\}$-matrices whose columns have no
  more than two nonzero entries. Such a matrix is totally unimodular
  if its rows can be divided into two partitions (one possibly empty)
  with the following condition. If two nonzeros in a column belong to
  the same partition, they must be of opposite signs, otherwise they
  must be in different row partitions. Consider $[\boundary_{p+1}]^T$,
  the transpose of $[\boundary_{p+1}]$. Each column of
  $[\boundary_{p+1}]^T$ contains at most two nonzero entries, and if
  there are two then they are of opposite signs because of the
  consistent orientations of the $(p+1)$-dimensional simplices. In
  this case, the simple division of rows into two partitions with one
  containing all rows and the other empty works. Thus
  $[\boundary_{p+1}]^T$ and hence $[\boundary_{p+1}]$ is totally
  unimodular.

  Now, reversing the orientation of a $(p+1)$-simplex means that the
  corresponding column of $[\boundary_{p+1}]$ be multiplied by $-1$.
  This column operation preserves the total unimodularity of
  $[\boundary_{p+1}]$. Since any arbitrary orientation of the
  $(p+1)$-simplices can be obtained by preserving or reversing their
  orientations in a consistent orientation, we have the result as
  claimed. 
\end{proof}

As a result of the above theorem and Theorem~\ref{thm:TU_plynmltm} we
have the following result.
\begin{corollary} 
  For a finite simplicial complex triangulating a $(p+1)$-dimensional
  compact orientable manifold, the optimal homologous chain problem
  can be solved for $p$-dimensional chains in polynomial time.
 \label{cor:manifold_plynml}
\end{corollary}
The result in Corollary~\ref{cor:manifold_plynml} when specialized to
$\mathbb{R}^{p+1}$ also appears in~\cite{Sullivan1990} though the
reasoning is different.

\subsection{Non-orientable manifolds} 
\label{subsec:nonorientable}

For non-orientable manifolds we give two examples which show that
total unimodularity may not hold in this case. We also
discuss the role of torsion in these examples in preparation 
for Theorem~\ref{thm:TUiff_trsnfr}. 

Our first example is the M\"obius strip and the second one is the
projective plane. Simplicial complexes for these two non-orientable
surfaces are shown in Figure~\ref{fig:nonorntbl}. The boundary
matrices $[\boundary_2]$ for these simplicial complexes are given in
the Appendix in~\eqref{mtx:moebius_b2} and~\eqref{mtx:prjctvpln_b2}.
\begin{figure}[ht]
  \centering
 \includegraphics[scale=0.85, trim=1.25in 8in 3.75in 1in, clip]
    {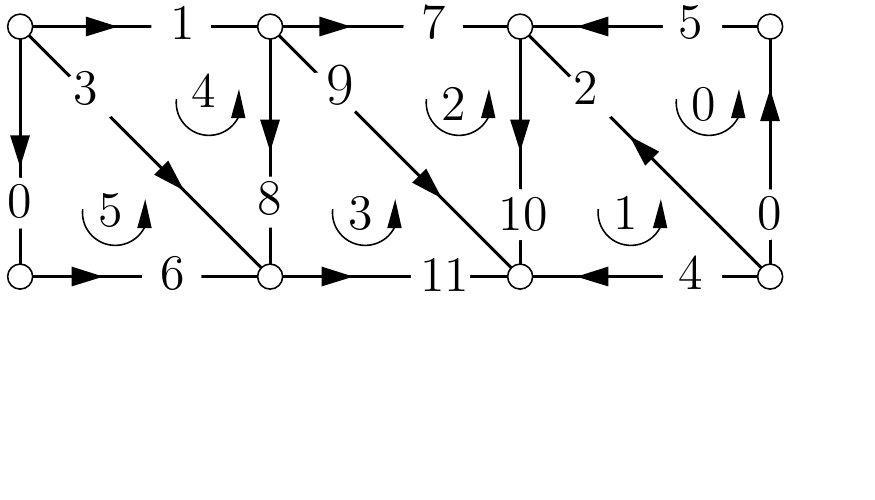}%
  \includegraphics[scale=0.85]
    {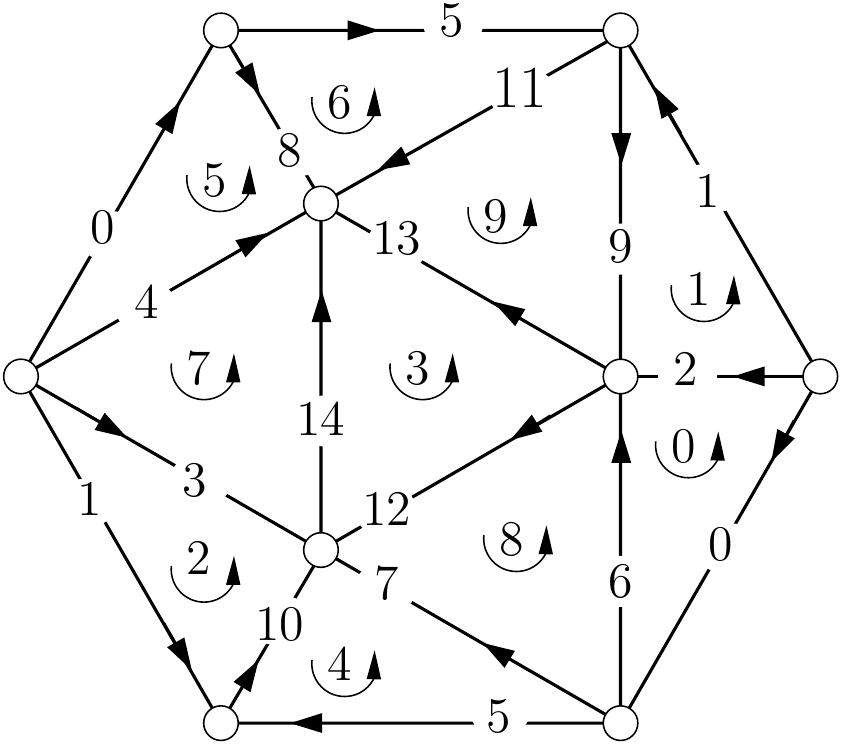}%
 \caption{Triangulations of two non-orientable manifolds, shown as
   abstract simplicial complexes. The left figure shows a
   triangulation of the M\"obius strip and the right one shows the
   projective plane. The numbers are the edge and triangles numbers.
   These correspond to the row and column numbers of the
   matrices~\eqref{mtx:moebius_b2} and~\eqref{mtx:prjctvpln_b2}.}
\label{fig:nonorntbl} 
\end{figure} 

Let $M$ be the M\"obius strip. We consider its absolute homology
$H_1(M)$ and its relative homology $H_1(M,\boundary M)$ relative to
its boundary. Consult \cite[page 135]{Munkres1984} to see how the
various homology groups are calculated using an exact sequence. We
note that $H_1(M) \cong \Z$, that is, its $H_1$ group has no
torsion. This can be seen by reducing the
matrix~\eqref{mtx:moebius_b2} in the Appendix to Smith normal form
(SNF). The SNF for the matrix consists of a $6\times 6$ identity
matrix on the top and a zero block below, which implies the absence of
torsion.

Let $K$ be the simplicial complex triangulating $M$. Consider a
submatrix $S$ of the matrix $[\boundary_2]$ shown in Appendix
as~\eqref{mtx:moebius_b2}. This submatrix is formed by selecting the
columns in the order 5, 4, 3, 2, 1, 0. From the matrix thus formed,
select the rows 0, 3, 8, 9, 10, 2 in that order. This selection of
rows and columns corresponds to all the triangles and the edges
encountered as one goes from left to right in the M\"{o}bius
triangulation shown in Figure~\ref{fig:nonorntbl}. The resulting
submatrix is
\[ 
S = \begin{bmatrix}
  \phantom{-}1&\phantom{-}0&\phantom{-}0&\phantom{-}0&
  \phantom{-}0&\phantom{-}1\\
  -1&\phantom{-}1&\phantom{-}0&\phantom{-}0&\phantom{-}0&
 \phantom{-}0\\
  \phantom{-}0&-1&\phantom{-}1&\phantom{-}0&\phantom{-}0&
  \phantom{-}0\\
  \phantom{-}0&\phantom{-}0&-1&\phantom{-}1&\phantom{-}0&
  \phantom{-}0\\
  \phantom{-}0&\phantom{-}0&\phantom{-}0&-1&\phantom{-}1&
  \phantom{-}0\\
  \phantom{-}0&\phantom{-}0&\phantom{-}0&\phantom{-}0&
  \phantom{-}1&-1\\
\end{bmatrix}
\]

The determinant of this matrix is $-2$ and this shows that the
boundary matrix is not totally unimodular. The SNF for this matrix, it
turns out, \emph{does} reveal the torsion. This matrix $S$ is the
relative boundary matrix $\rltvbndry{2}{L}{L_0}$ where $L = K$ and
$L_0$ are the edges in $\boundary M$. The SNF has 1's along the
diagonal and finally a 2. This is an example where there is no torsion
in the absolute homology but some torsion in the relative homology and
the boundary matrix is not totally unimodular. We formulate this
condition precisely in Theorem~\ref{thm:TUiff_trsnfr}.

The matrix $[\boundary_2]$ given in Appendix
as~\eqref{mtx:prjctvpln_b2} for the projective plane triangulation is
much larger. But it is easy to find a submatrix with determinant
greater than 1. This can be done by finding the M\"obius strip in the
triangulation of the projective plane. For example if one traverses
from top to bottom in the triangulation of the projective plane in
Figure~\ref{fig:nonorntbl} the triangles encountered correspond to
columns 6, 9, 3, 8, 4 of~\eqref{mtx:prjctvpln_b2} and the edges
correspond to rows 5, 11, 13, 12, 7. The corresponding submatrix is
\[
S =
\begin{bmatrix}
-1&\phantom{-}0&\phantom{-}0&\phantom{-}0&-1\\
-1&\phantom{-}1&\phantom{-}0&\phantom{-}0&\phantom{-}0\\
\phantom{-}0&-1&\phantom{-}1&\phantom{-}0&\phantom{-}0\\
\phantom{-}0&\phantom{-}0&-1&\phantom{-}1&\phantom{-}0\\
\phantom{-}0&\phantom{-}0&\phantom{-}0&-1&\phantom{-}1\\
\end{bmatrix}
\]
and its determinant is $-2$. Thus the boundary
matrix~\eqref{mtx:prjctvpln_b2} is not totally unimodular. Again, we
observe that there is relative torsion in $H_1(L,L_0)$ for the
subcomplexes corresponding to the selection of $S$ from
$[\boundary_2]$. Here $L$ consists of the triangles specified above,
which form a M\"obius strip in the projective plane. The subcomplex
$L_0$ consists of the edges forming the boundary of this strip. This
connection between submatrices and relative homology is examined in
the next section.

\section{Simplicial complexes} \label{sec:smplcl_cmplxs} 

Now we consider the more general case of simplicial complexes. Our
result in Theorem~\ref{thm:TUiff_trsnfr} characterizes the total
unimodularity of boundary matrices for arbitrary simplicial complexes.
Since we do not use any conditions about the geometric realization or
embedding in $\R^n$ for the complex, the result is also valid for
abstract simplicial complexes. As a corollary of the characterization
we show that the OHCP can be solved in polynomial time as long as
the input complex satisfies a torsion-related condition.

\subsection{Total unimodularity and relative torsion}
\label{sec:TUiff_trsnfr}
Let $K$ be a finite simplicial complex of dimension greater than $p$.
We will need to refer to its subcomplexes formed by the union of some
of its simplices of a specific dimension. This is formalized in the
definition below.
\begin{definition} \label{def:pure}
  A \emph{pure simplicial complex} of dimension $p$ is a simplicial
 complex formed by a collection of $p$-simplices and their proper
 faces. Similarly, a \emph{pure subcomplex} is a subcomplex that is a
 pure simplicial complex.
\end{definition}
An example of a pure simplicial complex of dimension $p$ is one that
triangulates a $p$-dimensional manifold. Another example, relevant to
our discussion, is a subcomplex formed by a collection of some
$p$-simplices of a simplicial complex and their lower dimensional
faces.

Let $L\subseteq K$ be a pure subcomplex of dimension $p+1$ and
$L_0\subset L$ be a pure subcomplex of dimension $p$. If
$[\boundary_{p+1}]$ is the matrix representing $\boundary_{p+1} :
C_{p+1}(K) \to C_p(K)$, then the matrix representing the relative
boundary operator
\[
\rltvbndry{p+1}{L}{L_0} : C_{p+1}(L,L_0) \to C_p(L, L_0)\, ,
\]
is obtained by first \emph{including} the columns of
$[\boundary_{p+1}]$ corresponding to $(p+1)$-simplices in $L$ and
then, from the submatrix so obtained, \emph{excluding} the rows
corresponding to the $p$-simplices in $L_0$ and any zero rows. The
zero rows correspond to $p$-simplices that are not faces of any of the
$(p+1)$-simplices of $L$.

As before, let $[\boundary_{p+1}]$ be the matrix of $\boundary_{p+1}$
in the elementary chain bases for $K$. Then the following holds.
\begin{theorem} \label{thm:TUiff_trsnfr} $[\boundary_{p+1}]$ is
  totally unimodular if and only if $H_p(L, L_0)$ is torsion-free,
  for all pure subcomplexes $L_0 , L$ of $K$ of dimensions $p$ 
  and $p+1$ respectively, where $L_0 \subset L$.
\end{theorem}
\begin{proof}
  $(\Rightarrow)$ We show that if $H_p(L,L_0)$ has torsion for some
  $L, L_0$ then $[\boundary_{p+1}]$ is not totally unimodular. Let
  $\left[\rltvbndry{p+1}{L}{L_0}\right]$ be the corresponding relative
  boundary matrix. Bring $\left[\rltvbndry{p+1}{L}{L_0}\right]$ to
  Smith normal form using the reduction algorithm
  \cite{Munkres1984}[pages 55--57]. This is a block matrix
  \[
  \begin{bmatrix}
     D & 0 \\
     0 & 0
  \end{bmatrix}
  \]
  where $D = \operatorname{diag}(d_1,\dots,d_l)$ is a diagonal matrix
  and the block row or column of zero matrices shown above may be
  empty, depending on the dimension of the matrix. Recall that $d_i$
  are integers and $d_i \ge 1$. Moreover, since $H_p(L, L_0)$ has
  torsion, $d_k > 1$ for some $1 \le k \le l$.  Thus the product $d_1
  \dots d_k$ is greater than 1. By a result of Smith \cite{Smith1861}
  quoted in \cite[page 50]{Schrijver1986}, this product is the greatest
  common divisor of the determinants of all $k \times k$ square
  submatrices of $\left[\rltvbndry{p+1}{L}{L_0}\right]$. But this
  implies that some square submatrix of
  $\left[\rltvbndry{p+1}{L}{L_0}\right]$, and hence of $[\boundary_{p+1}]$,
  has determinant magnitude greater than 1. Thus $[\boundary_{p+1}]$ is
  not totally unimodular.

  \bigskip\medskip
  \noindent $(\Leftarrow)$ Assume that $[\boundary_{p+1}]$ is not
  totally unimodular. We will show that then there exist 
  subcomplexes 
  $L_0$ and $L$ of dimensions $p$ and $(p+1)$ respectively, with $L_0
  \subset L$, such that $H_p(L, L_0)$ has torsion. Let $S$ be a square
  submatrix of $[\boundary_{p+1}]$ such that $\abs{\det(S)} > 1$. Let
  $L$ correspond to the columns of $[\boundary_{p+1}]$ that are
  \emph{included}  in $S$ and let $B_L$ be the submatrix of
  $[\boundary_{p+1}]$ formed by these columns. This submatrix $B_L$
  may contain zero rows. Those zero rows (if any) correspond to
  $p$-simplices that do not occur as a face of any of the
  $(p+1)$-simplices in $L$. In order to form $S$ from $B_L$, these
  zero rows can first be safely discarded to form a submatrix $B'_L$.
  This is because $\det(S) \ne 0$ and so these zero rows cannot occur
  in $S$. 

  The rows in $B'_L$ correspond to $p$-simplices that occur as a face
  of some $(p+1)$-simplex in $L$. Let $L_0$ correspond to rows of
  $B'_L$ which are \emph{excluded} to form $S$. Now $S$ is the matrix
  representation of the relative boundary matrix
  $\rltvbndry{p}{L}{L_0}$. Reduce $S$ to Smith normal form. The normal
  form is a square diagonal matrix. Since the elementary row and
  column operations preserve determinant magnitude, the determinant of
  the resulting diagonal matrix has magnitude greater than 1. Thus at
  least one of the diagonal entries in the normal form is greater than
  1. But then by \cite [page~61]{Munkres1984} $H_p(L, L_0)$ has
  torsion.
\end{proof}

\begin{remark}
  The characterization appears to be no easier to check than the
  definition of total unimodularity since it involves checking
  \emph{every} $L, L_0$ pair.  However, it is also no \emph{harder} to
  check than total unimodularity. This leads to the following result
  of possible interest in computational topology and matroid theory.
\end{remark}

\begin{corollary}
  For a simplicial complex $K$ of dimension greater than $p$, there is
  a polynomial time algorithm for answering the following question: Is
  $H_p(L,L_0)$ torsion-free for \emph{all} subcomplexes $L_0$ and $L$
  of dimensions $p$ and $(p+1)$ such that $L_0 \subset L$?
\end{corollary}

\begin{proof}
  Seymour's decomposition theorem for totally unimodular
  matrices \cite{Seymour1980},\cite[Theorem 19.6]{Schrijver1986}
  yields a polynomial time algorithm for deciding if a matrix is
  totally unimodular or not \cite[Theorem 20.3]{Schrijver1986}. That
  algorithm applied on the boundary matrix $[\boundary_{p+1}]$ proves
  the above assertion.
\end{proof}

\begin{remark}
  Note that the naive algorithm for the above problem is clearly
  exponential. For every pair $L, L_0$ one can use a polynomial time
  algorithm to find the Smith normal form. But the number of $L, L_0$
  pairs is exponential in the number of $p$ and $(p+1)$-simplices of
  $K$. 
\end{remark}

\begin{remark}
  The same polynomial time algorithm answers the
  question : Does $H_p(L,L_0)$ have torsion for \emph{some} pair
  $L, L_0$ ?
\end{remark}

\subsection{A special case}

In Section~\ref{sec:manifolds} we have seen the special case of
compact orientable manifolds. We saw that the top dimensional boundary
matrix of a finite triangulation of such a manifold is totally
unimodular. Now we show another special case for which the boundary
matrix is totally unimodular and hence OHCP is polynomial time
solvable. This case occurs when we ask for optimal $d$-chains in a
simplicial complex $K$ which is embedded in $\R^{d+1}$. In particular,
OHCP can be solved by linear programming for $2$-chains in
$3$-complexes embedded in $\R^3$. This follows from the following
result:
\begin{theorem} \label{thm:embedded}

  Let $K$ be a finite simplicial complex embedded in $\R^{d+1}$.
  Then, $H_d(L,L_0)$ is torsion-free for all pure subcomplexes $L_0$
  and $L$ of dimensions $d$ and $d+1$ respectively, such that $L_0
  \subset L$.
\end{theorem}
\begin{proof}
  We consider the $(d+1)$-dimensional relative \emph{cohomology group}
  $H^{d+1}(L,L_0)$ (See~\cite{Munkres1984} for example). It follows
  from the Universal Coefficient Theorem for cohomology~\cite[Theorem
  53.1]{Munkres1984} that
 \[
  H^{d+1}(L,L_0)=\Hom(H_{d+1}(L,L_0),\, \Z) \oplus 
  \Ext(H_d(L,L_0),\, \Z)
  \]
  where $\Hom$ is the group of all homomorphisms from $H_{d+1}(L,L_0)$
  to $\Z$ and $\Ext$ is the group of all of extensions between
  $H_d(L,L_0)$ and $\Z$. These definitions can be found
  in~\cite[Chapter 5 and 7]{Munkres1984}. The main observation is that
  if $H_d(L,L_0)$ has torsion, $\Ext(H_d(L,L_0),\Z)$ has torsion and
  hence $H^{d+1}(L,L_0)$ has torsion.

  On the other hand, by Alexander Spanier
  duality~\cite[page~296]{Spanier1966}
  \[ H^{d+1}(L,L_0)=H_0(\R^{d+1}\setminus \abs{L_0}, \R^{d+1}\setminus
  \abs{L})
  \] 
  where $\abs{L}$ denotes the underlying space of $L$.  Since
  $0$-dimensional homology groups cannot have torsion,
  $H^{d+1}(L,L_0)$ cannot have torsion. We reach a contradiction.
\end{proof}

\begin{corollary} 
  Given a $d$-chain $c$ in a weighted finite simplicial complex
  embedded in $\R^{d+1}$, an optimal chain homologous to $c$ can be
  computed by a linear program.
\end{corollary}
\begin{proof} Follows from Theorem~\ref{thm:embedded},
  Theorem~\ref{thm:TUiff_trsnfr}, and Theorem~\ref{cor:TU_ILP}.
\end{proof}

\subsection{Total unimodularity and M\"obius complexes}
\label{ssec:TUandMC}

As another special case, we provide a characterization of the total
unimodularity of $(p+1)$-boundary matrix of simplicial complexes in
terms of a forbidden complex called M\"{o}bius complex, for $p \leq
1$. In contrast to the previous characterization (in terms of relative
homology of $K$), we directly employ certain results on totally
unimodular matrices to derive this characterization in terms of
submatrices called cycle matrices. We show in
Theorem~\ref{thm:TUiff_noMCMs_2cplx} that the $(p+1)$-boundary matrix
of a finite simplicial complex for $p \leq 1$ is totally unimodular if
and only if the input complex does not have a $(p+1)$-dimensional
M\"{o}bius complex as a subcomplex. In particular, this observation
along with Theorem~\ref{thm:TUiff_trsnfr} implies that a $2$-complex
does not have relative torsion if and only if it does not have a
M\"{o}bius complex in it. We also demonstrate by example that this
result does not generalize to higher values of $p$.

\begin{definition}
   A $(p+1)$-dimensional \emph{cycle complex} is a sequence $\sigma_0,
   \dots,\sigma_{k-1}$ of $(p+1)$-simplices such that $\sigma_i$ and
   $\sigma_{j}$ have a common face if and only if $j=(i+1) \pmod{k}$
   and that common face is a $p$-simplex. Such a cycle complex
   triangulates a $(p+1)$-manifold. We call it a $(p+1)$-dimensional
   \emph{cylinder complex} if it is orientable and a
   $(p+1)$-dimensional \emph{M\"obius complex} if it is nonorientable.
\end{definition}


\begin{definition} \label{def:CM}
For $k \geq 2$, a $k \times k$ matrix $C$ is called a \emph{$k$-cycle 
matrix ($k$-CM)} if $C_{ij} \in \{-1,0,1\}$, and $C$ has the following
form up to row and column permutations and scalings by $-1$:
  \begin{equation} \label{mtx:cycle_matrix} 
    C = \begin{bmatrix} 1 & 0 & 0 & \cdots & 0 & 0 & \beta \\ 
      1 & 1 & 0 & \cdots & 0 & 0 &  0  \\ 
      0 & 1 & 1 & \cdots & 0 & 0 &  0  \\
      \vdots & \vdots & \vdots & \ddots &  \vdots & \vdots & \vdots \\ 
      0 & 0 & 0 & \cdots & 1 & 0 & 0\\
      0  & 0 & 0 & \cdots & 1  & 1 & 0 \\
      0  & 0 & 0 & \cdots & 0  & 1 & 1 
    \end{bmatrix},\;
    \beta = \pm 1.
  \end{equation}
A $k$-CM with $\beta = (-1)^k$ is termed a \emph{cylinder cycle
matrix} ($k$-CCM), while one with $\beta = (-1)^{k+1}$ is termed a
\emph{M\"obius cycle matrix} ($k$-MCM). We will refer to the form
shown in~\eqref{mtx:cycle_matrix} as the \emph{normal form} cycle
matrix.
\end{definition}

As an example, consider a triangulation $K$ of a M\"{o}bius strip with
$k\geq 5$ triangles shown in Figure \ref{fig:kmobius}. Let $K_0$ be
the complex for the boundary of the M\"obius strip. In the figure,
$K_0$ consists of the horizontal edges. Then the relative boundary
matrix $\left[\rltvbndry{2}{K}{K_0}\right]$ of the M\"{o}bius strip
$K$ modulo its edge $K_0$ is a $k$-MCM. The orientations of triangle
$\tau_{k-1}$ and that of the terminal edge $e_0$ are opposite if $k$
is even, but the orientations agree if $k$ is odd, giving $\beta =
(-1)^{k+1}$. Note that in Section~\ref{subsec:nonorientable}, the
submatrix $S$ of the boundary matrix of the M\"obius strip was such a
relative boundary matrix and it is an example of a 6-MCM. Another
example in that section was the 5-MCM obtained from the boundary
matrix of the projective plane.

Similarly, we observe a $k$-CCM as the relative boundary-$2$ matrix of
a cylinder triangulated with $k$ triangles, modulo the cylinder's
edges. Reversing the orientation of an edge or a triangle results in
scaling the corresponding row or column, respectively, of the boundary
matrix by $-1$. These examples motivate the names ``M\"obius'' and
``cylinder'' matrices -- a cycle matrix can be interpreted as the
relative boundary matrix of a M\"obius or cylinder complex. So, we
have the following result.
\begin{lemma} \label{lem:MCMcomplex} 
  Let $K$ be a finite simplicial complex of dimension greater than
  $p$. The boundary matrix $[\partial_{p+1}]$ has no $k$-{\em MCM} for
  any $k\geq 2$ if and only if $K$ does not have any
  $(p+1)$-dimensional M\"{o}bius complex as a subcomplex.
\end{lemma}

\begin{figure}[h] 
 \centering
 \includegraphics[scale=0.85]
 {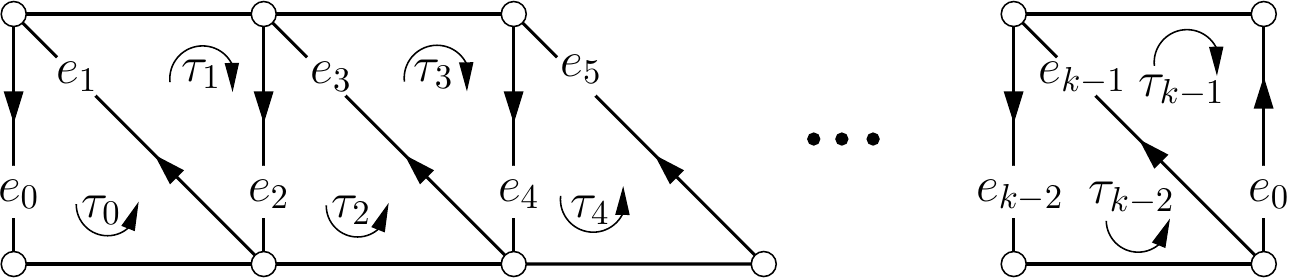}
\caption{Triangulation of a M\"{o}bius strip with $k$ triangles.}
 \label{fig:kmobius}
\end{figure}

It is now easy to see that the absence of M\"obius complexes is a
necessary condition for total unimodularity. We show that this
condition is also sufficient for $2$- or lower dimensional
complexes. We first need the simple result that an MCM is not totally
unimodular.

\begin{lemma}\label{lem:CMdtrmnt}
  Let $C$ be a $k$-CM. Then $\det C = 0$ if it is a $k$-CCM, and
  $\abs{\det C} = 2$ if it is a $k$-MCM.
\end{lemma}
\begin{proof}

   The matrix $C$ can always be brought into the normal form with a series of
    row and column exchanges and scalings by $-1$. 
   Note that these operations preserve the value of
  $\abs{\det C}$. Now assume that $C$ has been brought into the normal
  form and call that matrix $C'$. We expand along the first row of
  $C'$ to get $\det C' = 1 + (-1)^{k+1} \beta$, and the claim follows.
\end{proof}

\begin{theorem}
\label{thm:TUiff_noMCMs_2cplx}
For $p \leq 1$, $[\boundary_{p+1}]$ is totally unimodular if and only
if the simplicial complex $K$ has no M\"obius subcomplex of dimension
$p+1$.
\end{theorem}
\begin{proof} 
  \noindent ($\Rightarrow$) 
  If there is a M\"obius subcomplex of dimension $p+1$ in $K$, then by
  Lemma \ref{lem:MCMcomplex} an MCM appears as a submatrix of
  $[\boundary_{p+1}]$. That MCM is a certificate for
  $[\boundary_{p+1}]$ not being totally unimodular since its
  determinant has magnitude 2 by Lemma~\ref{lem:CMdtrmnt}.
 
  \medskip

  \noindent ($\Leftarrow$) 
  Let $K$ have no M\"obius subcomplexes of dimension $p+1$. Then by
  Lemma \ref{lem:MCMcomplex}, there are no MCMs as submatrices of
  $[\boundary_{p+1}]$.  Truemper~\cite[Theorem 28.3]{Truemper1992VII}
  has characterized all minimally non-totally unimodular matrices,
  i.e., matrices that are not totally unimodular, but their every
  proper submatrix is totally unimodular. These matrices belong to two
  classes, which Truemper denotes as $\mathscr{W}_1$ and
  $\mathscr{W}_7$. MCMs constitute the first class $\mathscr{W}_1$. A
  minimally non-totally unimodular matrix $W$ is in $\mathscr{W}_7$ if
  and only if $W$ has a row and a column containing at least four
  nonzeros each \cite[Cor.~28.5]{Truemper1992VII}. Since $p \leq 1$,
  no column of $[\boundary_{p+1}]$ can have four or more nonzeros, and
  hence no matrix from the class $\mathscr{W}_7$ can appear as a
  submatrix. Hence $[\boundary_{p+1}]$ is totally unimodular if $K$
  has no $(p+1)$-dimensional M\"obius subcomplexes.
\end{proof}
The necessary condition in Theorem~\ref{thm:TUiff_noMCMs_2cplx}
extends beyond $2$-complexes as Remark~\ref{extend} indicates. 
However, we cannot extend  the sufficiency
condition; Remark~\ref{not-extend} presents a counterexample.
\begin{remark}
  Note that the absence of M\"obius subcomplexes is a necessary
  condition for $[\boundary_{p+1}]$ to be totally unimodular for {\em
    all} $p$.  More precisely, if the simplicial complex $K$ of
  dimension greater than $p$ has a M\"obius subcomplex of dimension
  $p+1$ then $[\boundary_{p+1}]$ is not totally unimodular. By Lemma
  \ref{lem:MCMcomplex}, an MCM appears as a submatrix of
  $[\boundary_{p+1}]$ in this case. Its determinant has magnitude 2 by
  Lemma~\ref{lem:CMdtrmnt}, trivially certifying that
  $[\boundary_{p+1}]$ is not totally unimodular.
\label{extend}
\end{remark}

\begin{remark}
  The characterization in Theorem \ref{thm:TUiff_noMCMs_2cplx} does
  not hold for higher values of $p$. We present a $3$-complex which
  does not have a $3$-dimensional M\"obius subcomplex, but whose
  $[\boundary_{3}]$ is not totally unimodular. Consider the simplicial
  complex consisting of the following seven tetrahedra formed from
  seven points numbered $0$--$6$: $\, (0,1,2,3),\, (0,1,2,4),\,
  (0,1,2,5),\, (0,1,2,6),\, (0,1,3,4),\, (0,2,3,5),\, (1,2,3,6)$. It
  can be verified that the $19 \times 7$ boundary matrix
  $[\boundary_{3}]$ of this simplicial complex has the $7 \times 7$
  matrix
\[ 
W = \begin{bmatrix}
  -1&-1&-1&-1&\phantom{-}0&\phantom{-}0&\phantom{-}0\\
  \phantom{-}1&\phantom{-}0&\phantom{-}0&\phantom{-}0&-1&
  \phantom{-}0&\phantom{-}0\\
  -1&\phantom{-}0&\phantom{-}0&\phantom{-}0&\phantom{-}0&-1&\phantom{-}0\\
  \phantom{-}1&\phantom{-}0&\phantom{-}0&\phantom{-}0&\phantom{-}0&
  \phantom{-}0&-1\\
  \phantom{-}0&\phantom{-}1&\phantom{-}0&\phantom{-}0&\phantom{-}1&
  \phantom{-}0&\phantom{-}0\\
  \phantom{-}0&\phantom{-}0&-1&\phantom{-}0&\phantom{-}0&\phantom{-}1&
  \phantom{-}0\\
  \phantom{-}0&\phantom{-}0&\phantom{-}0&\phantom{-}1&\phantom{-}0&
  \phantom{-}0&\phantom{-}1\\
\end{bmatrix}
\]
as a submatrix where $\det(W)=-2$, certifying that $[\boundary_{3}]$
is not totally unimodular. In fact, $W$ is the {\em only} submatrix of
$[\boundary_{3}]$ which is not totally unimodular, and it belongs to
the class $\mathscr{W}_7$ of minimally non-totally unimodular
matrices. 
\label{not-extend}
\end{remark}
 
\section{Experimental Results}

We have implemented our linear programming method to solve the optimal
homologous chain problem. In Figure~\ref{fig:trs_sldannls} we show
some results of preliminary experiments.

The top row in Figure~\ref{fig:trs_sldannls} shows the computation of
optimal homologous 1-chains on the simplicial complex representation
of a torus. The longer chain in each torus figure is the initial chain
and the tighter shorter chain is the optimal homologous chain computed
by our algorithm. The bottom row shows the result of the computation
of an optimal 2-chain on a simplicial complex of dimension 3. The
complex is the tetrahedral triangulation of a solid annulus -- a solid
ball from which a smaller ball has been removed. Two cut-away views
are shown. The outer surface of the sphere is the initial chain and
the inner surface is computed as the optimal 2-chain.

In these experiments we used the linear program~\eqref{opt:LP}. The
initial chains used had values in $\{-1,0,1\}$ on the simplices. In
the torus examples for instance, the initial chain was 0 everywhere
except along the initial curve shown. The curve was given an arbitrary
orientation and the values of the chain on the edges forming the curve
were $+1$ or $-1$ depending on the edge orientation. In these
examples, the resulting optimal chains were oriented curves, with
values of $\pm 1$ on the edges along the curve. This is by no means
guaranteed theoretically, as seen in the hour glass example in
Remark~\ref{rem:hour_glass}. The only guarantee is that of
integrality. However if it is essential that the optimal chain has
values only in $\{-1,0,1\}$ then the linear program~\eqref{opt:LP_L0}
or it's Euclidean variant can be used, imposing the additional
constraint on the values of the optimal solution $\xx$ as shown in
linear program~\eqref{opt:LP_L0}.

\begin{figure}[ht!]
  \centering
  \includegraphics[scale=0.25, trim=2.75in 1.5in 2.75in 2.65in, clip]
  {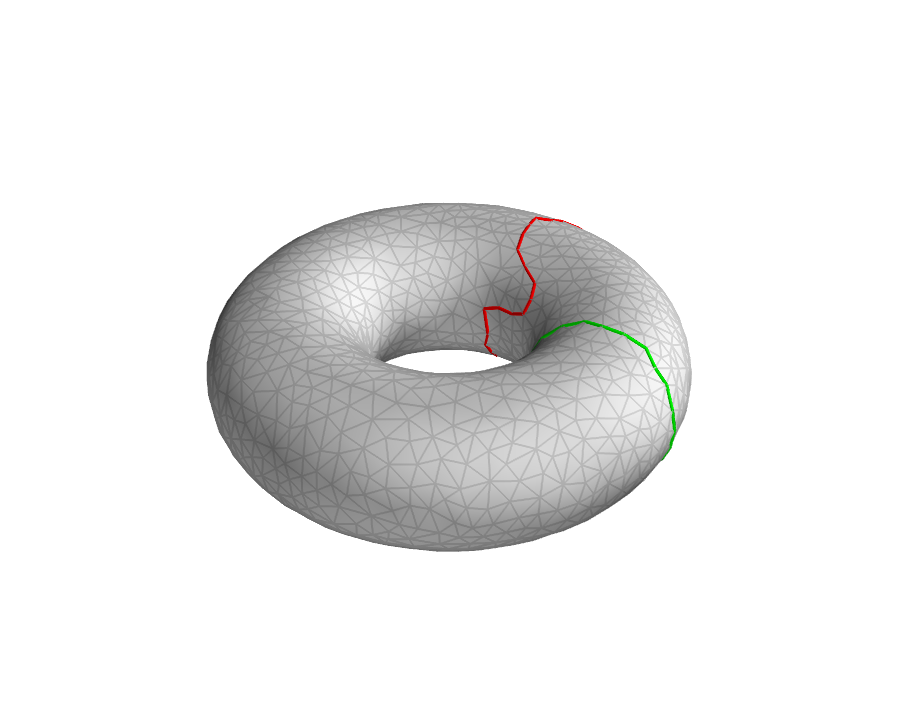}
  \includegraphics[scale=0.25, trim=2.75in 1.5in 2.75in 2.65in, clip]
  {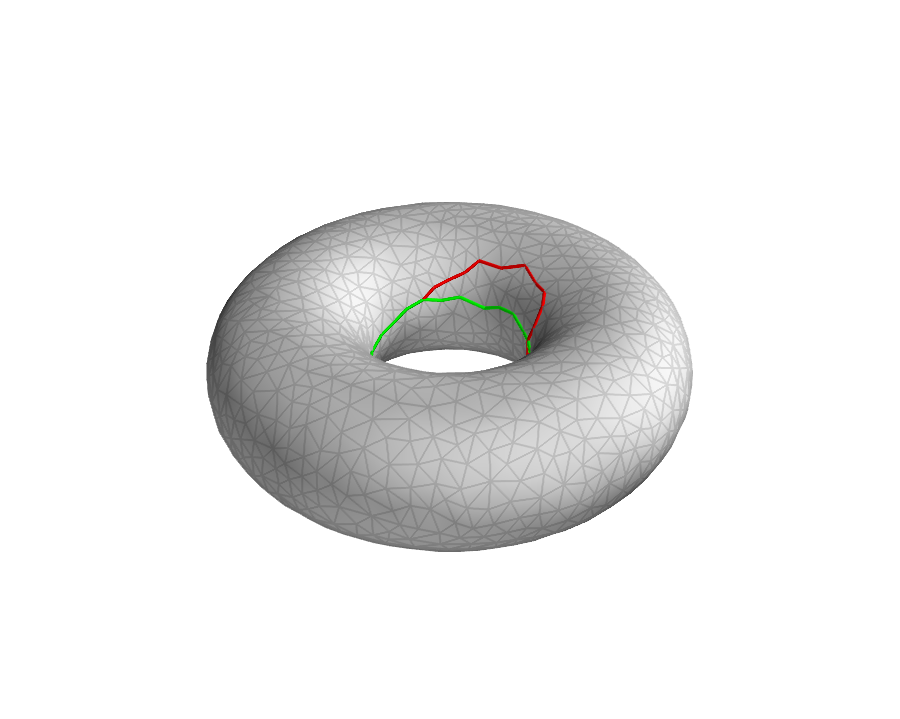}\\
 \includegraphics[trim=0in 1in 1.5in 1.5in, clip, scale=0.25]
  {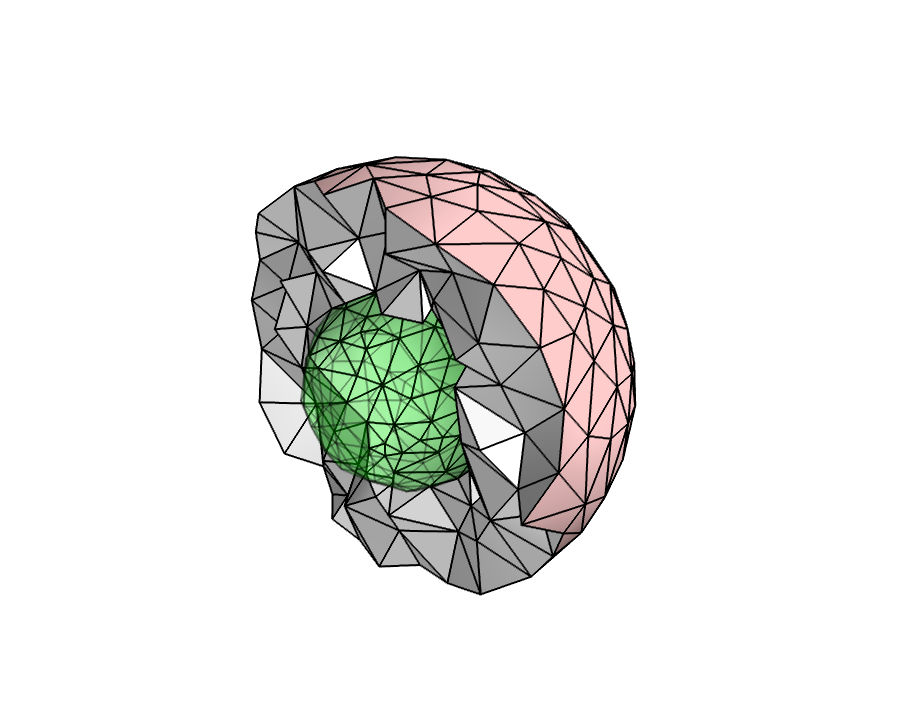}
  \includegraphics[trim=2.5in 1in 1.5in 1.5in, clip, scale=0.25]
  {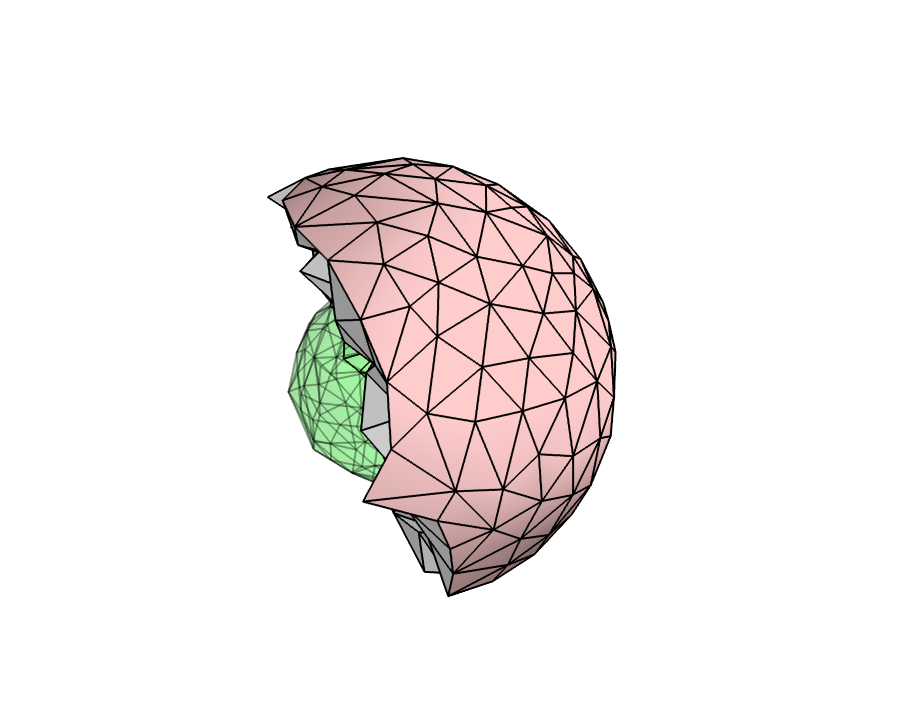}
  \caption{Some experimental results. See text for details on what is
    being computed here.}
\label{fig:trs_sldannls}
\end{figure}

\section{Discussion}

Several questions crop up from our problem formulation and results.
Instead of 1-norm $\norm{W\, \xx}_1$, we can consider minimizing
$\sum_i w_i\,x_i$. In this case, the weights appear with signs and
solutions may be unbounded. Nevertheless, our result in
Theorem~\ref{thm:TU_plynmltm} remains valid. Of course, in this case
we do not need to introduce $x_i^+$ and $x_i^-$ since the objective
function uses $x_i$ rather than $\abs{x_i}$. We may introduce more
generalization in the OHCP formulation by considering a general matrix
$W$ instead of requiring it to be diagonal and then asking for
minimizing $\norm{W\,\xx}_1$. We do not know if the corresponding
optimization problem can be solved by a linear program. Can this
optimization problem be solved in polynomial time for some interesting
classes of complexes?

We showed that OHCP under $\Z$ coefficients can be solved by linear
programs for a large class of topological spaces that have no relative
torsion. This leaves a question for the cases when there is relative
torsion. Is the problem NP-hard under such constraint? Taking the cue
from our results, one can also ask the following question. Even though
we know that the problem is NP-hard under $\Z_2$ coefficients, is it
true that OHCP in this case is polynomial time solvable at least for
simplicial complexes that have no relative torsions (considered under
$\Z$)? The answer is negative since OHCP for surfaces in
$\mathbb{R}^3$ is NP-hard under $\Z_2$ coefficients~\cite{ChErNa2009}
even though they are known to be torsion-free.

Even if the input complex has relative torsion, the constraint
polyhedron of the linear program may still have vertices with integer
coordinates. In that case, the linear program may still give an
integer solution for chains that steer the optimization path toward
such a vertex. In fact, we have observed experimentally that, for some
$2$-complexes with relative torsion, the linear program finds the
integer solution for some input chains. It would be nice to
characterize the class of chains for which the linear program still
provides a solution even if the input complex has relative torsion.

A related question that has also been investigated recently is the
problem of computing an optimal homology basis from a given complex.
Again, positive results have been found for low dimensional cases such
as surfaces~\cite{ErWh2005} and one dimensional homology for
simplicial complexes~\cite{ChFr2010,DSW2010}. The result of Chen and
Freedman~\cite{ChFr2010a} implies that even this problem is NP-hard
for high dimensional cycles under $\Z_2$. What about $\Z$? As in OHCP,
would we have any luck here?

\vspace{0.1in}
\noindent {\bf Acknowledgments.} We acknowledge the helpful
discussions with Dan Burghelea from OSU mathematics department and
thank Steven Gortler for pointing out the result in John Sullivan's
thesis. Tamal Dey acknowledges the support of NSF grants CCF-0830467
and CCF-0915996. The research of Anil Hirani is funded by NSF CAREER
Award, Grant No. DMS-0645604. We acknowledge the opportunity provided
by NSF via a New Directions Short Course at the Institute for
Mathematics and its Applications (IMA) which initiated the present
collaboration of the authors.

\bibliographystyle{acmurldoi} 
\bibliography{homology}

\begin{thebibliography}{10}

\bibitem{BeTs1997}
{\sc Bertsimas, D., and Tsitsiklis, J.~N.}
\newblock {\em Introduction to Linear Optimization}.
\newblock Athena Scientific, Belmont, MA., 1997.

\bibitem{ChVeErLaWh2008}
{\sc Chambers, E.~W., {Colin de Verdi\`{e}re \'{E}.}, Erickson, J., Lazarus,
  F., and Whittlesey, K.}
\newblock Splitting (complicated) surfaces is hard.
\newblock {\em Comput. Geom. Theory Appl. 41\/} (2008), 94--110.

\bibitem{ChErNa2009}
{\sc Chambers, E.~W., Erickson, J., and Nayyeri, A.}
\newblock Minimum cuts and shortest homologous cycles.
\newblock In {\em SCG '09: Proc. 25th Ann. Sympos. Comput. Geom.\/} (2009),
  pp.~377--385.

\bibitem{ChFr2010a}
{\sc Chen, C., and Freedman, D.}
\newblock Hardness results for homology localization.
\newblock In {\em SODA '10: Proc. 21st Ann. ACM-SIAM Sympos. Discrete
  Algorithms\/} (2010), pp.~1594--1604.

\bibitem{ChFr2010}
{\sc Chen, C., and Freedman, D.}
\newblock Measuring and computing natural generators for homology groups.
\newblock {\em Computational Geometry 43}, 2 (2010), 169--181.
\newblock Special Issue on the 24th European Workshop on Computational Geometry
  (EuroCG'08).

\bibitem{VeEr2006}
{\sc {Colin de Verdi\`{e}re \'{E}.}, and Erickson, J.}
\newblock Tightening non-simple paths and cycles on surfaces.
\newblock In {\em SODA '06: Proc. 17th Ann. ACM-SIAM Sympos. Discrete
  Algorithms\/} (2006), pp.~192--201.

\bibitem{SiGh2007}
{\sc de~Silva, V., and Ghrist, R.}
\newblock Homological sensor networks.
\newblock {\em Notices of the American Mathematical Society 54}, 1 (2007),
  10--17.

\bibitem{DeLiSuCo2008}
{\sc Dey, T.~K., Li, K., Sun, J., and Cohen-Steiner, D.}
\newblock Computing geometry-aware handle and tunnel loops in 3d models.
\newblock In {\em SIGGRAPH '08: ACM SIGGRAPH 2008 papers\/} (New York, NY, USA,
  2008), pp.~1--9.

\bibitem{DSW2010}
{\sc Dey, T.~K., Sun, J., and Wang, Y.}
\newblock Approximating loops in a shortest homology basis from point data.
\newblock In {\em SCG '10: Proc. 26th Ann. Sympos. Comput. Geom.\/} (2010),
  pp.~166--175.

\bibitem{EdLeZo2002}
{\sc Edelsbrunner, H., Letscher, D., and Zomorodian, A.}
\newblock Topological persistence and simplification.
\newblock {\em Discrete Comput. Geom. 28\/} (2002), 511--533.

\bibitem{ErWh2005}
{\sc Erickson, J., and Whittlesey, K.}
\newblock Greedy optimal homotopy and homology generators.
\newblock In {\em SODA '05: Proc. 16th Ann. ACM-SIAM Sympos. Discrete
  Algorithms\/} (2005), pp.~1038--1046.

\bibitem{GuHeRoTeTs1993}
{\sc G\"uler, O., den Hertog, D., Roos, C., Terlaky, T., and Tsuchiya, T.}
\newblock Degeneracy in interior point methods for linear programming: a
  survey.
\newblock {\em Annals of Operations Research 46-47}, 1 (March 1993), 107--138.

\bibitem{HeTo1956}
{\sc Heller, I., and Tompkins, C.~B.}
\newblock An extension of a theorem of {D}antzig's.
\newblock In {\em Linear Inequalities and Related Systems}, H.~W. Kuhn and
  A.~W. Tucker, Eds., Annals of Mathematics Studies, no. 38. Princeton
  University Press, Princeton, N. J., 1956, pp.~247--254.

\bibitem{Munkres1984}
{\sc Munkres, J.~R.}
\newblock {\em Elements of Algebraic Topology}.
\newblock Addison--Wesley Publishing Company, Menlo Park, 1984.

\bibitem{Schrijver1986}
{\sc Schrijver, A.}
\newblock {\em Theory of Linear and Integer Programming}.
\newblock Wiley-Interscience Series in Discrete Mathematics. John Wiley \& Sons
  Ltd., Chichester, 1986.
\newblock A Wiley-Interscience Publication.

\bibitem{Seymour1980}
{\sc Seymour, P.~D.}
\newblock Decomposition of regular matroids.
\newblock {\em J. Combin. Theory Ser. B 28}, 3 (1980), 305--359.

\bibitem{Smith1861}
{\sc Smith, H. J.~S.}
\newblock On systems of linear indeterminate equations and congruences.
\newblock {\em Philosophical Transactions of the Royal Society of London 151\/}
  (1861), 293--326.

\bibitem{Spanier1966}
{\sc Spanier, E.~H.}
\newblock {\em Algebraic Topology}.
\newblock McGraw-Hill Book Co., New York, 1966.

\bibitem{Sullivan1990}
{\sc Sullivan, J.~M.}
\newblock {\em A Crystalline Approximation Theorem for Hypersurfaces}.
\newblock PhD thesis, Princeton University, 1990.

\bibitem{TaJa2009}
{\sc Tahbaz-Salehi, A., and Jadbabaie, A.}
\newblock Distributed coverage verification algorithms in sensor networks
  without location information.
\newblock {\em IEEE Transactions on Automatic Control 55}, 8 (2010), to appear.

\bibitem{Tardos1985}
{\sc Tardos, E.}
\newblock A strongly polynomial minimum cost circulation algorithm.
\newblock {\em Combinatorica 5}, 3 (September 1985), 247--255.

\bibitem{Tardos1986}
{\sc Tardos, E.}
\newblock A strongly polynomial algorithm to solve combinatorial linear
  programs.
\newblock {\em Operations Research 34}, 2 (March 1986), 250--256.

\bibitem{Truemper1992VII}
{\sc Truemper, K.}
\newblock A decomposition theory for matroids. {VII}. analysis of minimal
  violation matrices.
\newblock {\em Journal of Combinatorial Theory, Series B 55}, 2 (1992),
  302--335.

\bibitem{VeDa1968}
{\sc {Veinott, {Jr.}, A. F.}, and Dantzig, G.~B.}
\newblock Integral extreme points.
\newblock {\em SIAM Review 10}, 3 (1968), 371--372.

\end{thebibliography}

\pagebreak[4]

\section*{Appendix}
\subsection*{Boundary matrices for non-orientable surfaces}

The boundary matrices $[\boundary_2]$ for the M\"obius strip and
projective plane triangulations shown in Figure~\ref{fig:nonorntbl}
are given below. The row numbers are edge numbers and the column
numbers are triangle numbers which are displayed in
Figure~\ref{fig:nonorntbl}. 

\begin{center}
  $[\boundary_2]$ for M\"obius strip :
\end{center}
\begin{equation}
\label{mtx:moebius_b2}
\begin{bmatrix}
&0: &1: &2: &3: &4: &5:\\
\\
0: &\phantom{-}1&\phantom{-}0&\phantom{-}0&\phantom{-}0&
\phantom{-}0&\phantom{-}1\\
1: &\phantom{-}0&\phantom{-}0&\phantom{-}0&\phantom{-}0&-1&
\phantom{-}0\\
2:&-1&\phantom{-}1&\phantom{-}0&\phantom{-}0&\phantom{-}0&
\phantom{-}0\\
3:&\phantom{-}0&\phantom{-}0&\phantom{-}0&\phantom{-}0&
\phantom{-}1&-1\\
4:&\phantom{-}0&-1&\phantom{-}0&\phantom{-}0&\phantom{-}0&
\phantom{-}0\\
5:&\phantom{-}1&\phantom{-}0&\phantom{-}0&\phantom{-}0&
\phantom{-}0&\phantom{-}0\\
6:&\phantom{-}0&\phantom{-}0&\phantom{-}0&\phantom{-}0&
\phantom{-}0&\phantom{-}1\\
7:&\phantom{-}0&\phantom{-}0&-1&\phantom{-}0&\phantom{-}0&
\phantom{-}0\\
8:&\phantom{-}0&\phantom{-}0&\phantom{-}0&\phantom{-}1&-1&
\phantom{-}0\\
9:&\phantom{-}0&\phantom{-}0&\phantom{-}1&-1&\phantom{-}0&
\phantom{-}0\\
10:&\phantom{-}0&\phantom{-}1&-1&\phantom{-}0&\phantom{-}0&
\phantom{-}0\\
11:&\phantom{-}0&\phantom{-}0&\phantom{-}0&\phantom{-}1&
\phantom{-}0&\phantom{-}0\\
\end{bmatrix}
\end{equation}

\bigskip
\bigskip
\begin{center}
$[\boundary_2]$ for projective plane :
\end{center}
\setcounter{MaxMatrixCols}{12}
\begin{equation}
\label{mtx:prjctvpln_b2}
\begin{bmatrix}
&0: &1: &2: &3: &4: &5: &6: &7: &8: &9:\\
\\
0:&-1&\phantom{-}0&\phantom{-}0&\phantom{-}0&\phantom{-}0&-1&
\phantom{-}0& \phantom{-}0&\phantom{-}0&\phantom{-}0\\
1: &\phantom{-}0&\phantom{-}1&\phantom{-}1&\phantom{-}0&
\phantom{-}0&\phantom{-}0&\phantom{-}0&\phantom{-}0&\phantom{-}0&
\phantom{-}0\\
2:&\phantom{-}1&-1&\phantom{-}0&\phantom{-}0&\phantom{-}0&
\phantom{-}0& \phantom{-}0&\phantom{-}0&\phantom{-}0&\phantom{-}0\\
3:&\phantom{-}0&\phantom{-}0&-1&\phantom{-}0&\phantom{-}0&
\phantom{-}0& \phantom{-}0&\phantom{-}1&\phantom{-}0&\phantom{-}0\\
4:&\phantom{-}0&\phantom{-}0&\phantom{-}0&\phantom{-}0&
\phantom{-}0&\phantom{-}1&\phantom{-}0&-1&\phantom{-}0&
\phantom{-}0\\
5:&\phantom{-}0&\phantom{-}0&\phantom{-}0&\phantom{-}0&-1&
\phantom{-}0&-1& \phantom{-}0&\phantom{-}0&\phantom{-}0\\
6:&-1&\phantom{-}0&\phantom{-}0&\phantom{-}0&\phantom{-}0&
\phantom{-}0& \phantom{-}0&\phantom{-}0&\phantom{-}1&\phantom{-}0\\
7:&\phantom{-}0&\phantom{-}0&\phantom{-}0&\phantom{-}0&
\phantom{-}1&\phantom{-}0&\phantom{-}0&\phantom{-}0&-1&
\phantom{-}0\\
8:&\phantom{-}0&\phantom{-}0&\phantom{-}0&\phantom{-}0&
\phantom{-}0&-1&\phantom{-}1&\phantom{-}0&\phantom{-}0&
\phantom{-}0\\
9:&\phantom{-}0&\phantom{-}1&\phantom{-}0&\phantom{-}0&
\phantom{-}0&\phantom{-}0&\phantom{-}0&\phantom{-}0&
\phantom{-}0&-1\\
10:&\phantom{-}0&\phantom{-}0&\phantom{-}1&\phantom{-}0&-1&
\phantom{-}0& \phantom{-}0&\phantom{-}0&\phantom{-}0&
\phantom{-}0\\
11:&\phantom{-}0&\phantom{-}0&\phantom{-}0&\phantom{-}0&
\phantom{-}0& \phantom{-}0&-1&\phantom{-}0&\phantom{-}0&
\phantom{-}1\\
12:&\phantom{-}0&\phantom{-}0&\phantom{-}0&-1&\phantom{-}0&
\phantom{-}0&\phantom{-}0&\phantom{-}0&\phantom{-}1&\phantom{-}0\\
13:&\phantom{-}0&\phantom{-}0&\phantom{-}0&\phantom{-}1&
\phantom{-}0&\phantom{-}0&\phantom{-}0&\phantom{-}0&
\phantom{-}0&-1\\
14:&\phantom{-}0&\phantom{-}0&\phantom{-}0&-1&\phantom{-}0&
\phantom{-}0&\phantom{-}0&\phantom{-}1&\phantom{-}0&\phantom{-}0\\
\end{bmatrix}
\end{equation}

\end{document}